\documentclass[11pt]{article}
\usepackage{amsfonts}
\usepackage{amssymb,amsfonts,amsmath,amsthm,cite}
\usepackage{enumerate}

\usepackage{graphicx}
\usepackage{tikz}
\usetikzlibrary{calc,decorations.pathreplacing}
\usetikzlibrary{arrows,shapes,positioning}
\usetikzlibrary{decorations.markings}

\tikzset{middlearrow/.style={
        decoration={markings,
            mark= at position 0.5 with {\arrow{#1}} ,
        },
        postaction={decorate}
    }
}

\allowdisplaybreaks[4]

\theoremstyle{plain}
\newtheorem{theorem}{Theorem}[section]
\newtheorem{lemma}[theorem]{Lemma}
\newtheorem{corollary}[theorem]{Corollary}

\theoremstyle{definition}

\newtheorem{example}[theorem]{Example}

\theoremstyle{remark}
\newtheorem{remark}[theorem]{Remark}

\makeatletter \@addtoreset{equation}{section}
\begin{document}

\begin{center}
{\Large {\bf Schmidt Type Partitions}}

\vskip 6mm

{\small Runqiao Li$^{1}$ and Ae Ja Yee$^{2}$\footnote[2]{Partially supported by a grant ($\#633963$) from the Simons Foundation.} \\[2mm]

Department of Mathematics\\
The Pennsylvania State University\\
University Park, PA 16802, USA\\[3mm]

$^{1}$runqiaoli@outlook.com and $^{2}$yee@psu.edu\\[2mm]}
\end{center}

\noindent {\bf Abstract.}
Recently, Andrews and Paule studied Schmidt type partitions using MacMahon's Partition Analysis and obtained various interesting results. In this paper, we focus on the combinatorics of Schmidt type partition theorems and characterize them in a general and refined form. In addition, we also present some overpartition analogues of Schmidt type partition theorems. %All results are proved bijectively.



\section{Introduction}

A \emph{partition} is a finite weakly decreasing sequence of positive integers  $\lambda=(\lambda_1,\lambda_2,\ldots,\lambda_l)$. Each summand $\lambda_i$ is called a \emph{part} of the partition $\lambda$ and the number of parts is called the \emph{length} of $\lambda$, denoted by $\ell(\lambda)$. The \emph{weight} of $\lambda$ is the sum of its parts, denoted by $|\lambda|$. If $|\lambda|=n$, $\lambda$ is called a partition of $n$. To distinguish from other types of partitions, partitions with no restrictions on parts are referred to ordinary partitions if needed.

In 1999, Schmidt \cite{Sch99} proposed the following interesting problem. 
\begin{quote}
ÄúLet $s(n)$ denote the number of partitions $(\lambda_1,\lambda_2,\lambda_3\ldots)$ satisfying $\lambda_1>\lambda_2>\lambda_3>\cdots$ and $\lambda_1+\lambda_3+\lambda_5+\cdots=n$. Then $s(n)$ equals the number of all partitions of $n$ for $n\geq1$.
\end{quote}

This problem was immediately confirmed by Mork \cite{M20} with eight other independent solvers. However, the problem has not received much attention until a recent work of Andrews and Paule. In 2022, Andrews and Paule \cite{AP22} explored Schmidt's problem using MacMahon's Partition Analysis, which leads to numerous interesting results. %We first introduce some necessary definitions. %notions and notation.

A \emph{Schmidt $k$-partition} of $n$ is defined to be a partition $(\lambda_1,\lambda_2,\lambda_3,\ldots)$ with $\lambda_1>\lambda_2>\lambda_3>\cdots$ and $\lambda_1+\lambda_{k+1}+\lambda_{2k+1}+\cdots=n$; an \emph{unrestricted Schmidt $k$-partition} of $n$ is defined to be a partition $(\lambda_1,\lambda_2,\lambda_3,\ldots)$ with $\lambda_1\geq\lambda_2\geq\lambda_3\geq\cdots$ and $\lambda_1+\lambda_{k+1}+\lambda_{2k+1}+\cdots=n$.

Schmidt's result can be restated as follows.
\begin{theorem}\label{gp1}
The number of Schmidt $2$-partitions of $n$ equals the number of all partitions of $n$.
\end{theorem}

\begin{remark}
Wang and Xu \cite{WX22} found a strong refinement of Theorem \ref{gp1} according to a partition statistic, namely the minimal excludant.
\end{remark}

A $k$-tuple $(\alpha^1,\alpha^2,\ldots,\alpha^k)$ is called a \emph{$k$-tuple partition} of $n$ if each $\alpha^i$ is a partition and $|\alpha^1|+|\alpha^2|+\cdots+|\alpha^k|=n$. We adopt the terminologies bipartitions and triple partitions for $2$-tuple partitions and $3$-tuple partitions, respectively. A partition $\lambda$ is called \emph{$k$-distinct} if two adjacent parts differ by at least $k$. For the special case when $k=1$, we call strict partitions for short. %Andrews and Paule\cite{AP22} observed that there exists a relationship between Schmidt $3$-partitions and a certain type of triple partitions.

In \cite{AP22}, generalizing the result of Schmidt, Andrews and Paule established some connections between Schmidt $k$-partitions and bipartitions or tripartitions of partitions for $k=2,3$. 

\begin{theorem}[Andrews--Paule \cite{AP22}] \label{gp2}
The number of unrestricted Schmidt $2$-partitions of $n$ equals the number of bipartitions of $n$.
\end{theorem}

%A partition $\lambda$ is called \emph{$k$-distinct} if two adjacent parts differ by at least $k$. For the special case when $k=1$, we call stinct partitions for short. Andrews and Paule\cite{AP22} observed that there exists a relationship between Schmidt $3$-partitions and a certain type of triple partitions.
\begin{theorem}[Andrews--Paule \cite{AP22}] \label{gp3}
Schmidt $3$-partitions of $n$ with $3t$ or $3t-1$ parts are equinumerous with triple partitions $(\alpha,\beta,\gamma)$ of $n$ where $\alpha$ is a $2$-distinct partition of length $t$, $\beta$ is a strict partition of length $t$ and $\gamma$ is an ordinary partition of length at most $t$.
\end{theorem}

\begin{remark}
Recently, Ji \cite{Ji21} gave a combinatorial proof of Theorem \ref{gp2}, from which a four-variable refinement follows. 
\end{remark}

In \cite{AP22}, Andrews and Paule viewed partition diamonds as Schmidt type partitions and investigated the arithmetic properties of $k$-elongated partition diamonds. To give the precise definition of $k$-elongated partition diamonds, we first introduce an order $\succeq$ on integer sequences. Given two sequences $(a_1,a_2,\ldots,a_m)$ and $(b_1,b_2,\ldots,b_n)$, we say
\[
(a_1,a_2,\ldots,a_m)\succeq(b_1,b_2,\ldots,b_n) \text{ if and only if }  \min\{a_1, \ldots,a_m\}\geq \max\{b_1,\ldots,b_n\}.
\]
%if and only if $\min\{a_1,a_2,\ldots,a_m\}\geq \max\{b_1,b_2,\ldots,b_n\}$. 
Also, for a  sequence $(a)$ of length $1$, we will write it as $a$. %, we do not distinguish between the integer $a$ and the monuple $(a)$.

A \emph{$k$-elongated partition diamond} of $n$ is a nonnegative integer sequence $\pi=(\pi_1, \ldots,\pi_{t(2k+1)})$  for some $t\ge 1$ such that 
\begin{enumerate}
\item[(i)]
$\pi_{1}+\pi_{2k+2}+\pi_{4k+3}+\cdots\pi_{(t-1)(2k+1)+1}=n;$
\item[(ii)] for $i=0,\ldots, t-1$,
\end{enumerate}
\begin{equation}
\pi_{i(2k+1)+1} \succeq(\pi_{i(2k+1)+2},\pi_{i(2k+1)+3} )\succeq\cdots\succeq(\pi_{(i+1)(2k+1)-1},\pi_{(i+1)(2k+1)}) \succeq \pi_{(i+1)(2k+1)+1},\label{buildingblock}
\end{equation}
where $\pi_{t(2k+1)+1}=0$. 
\iffalse
\begin{align*}
&\pi_1\succeq(\pi_2,\pi_3)\succeq(\pi_4,\pi_5)\succeq\cdots\succeq(\pi_{2k},\pi_{2k+1})
\succeq\pi_{2k+2},\\
&\pi_{2k+2}\succeq(\pi_{2k+3},\pi_{2k+4})\succeq(\pi_{2k+5},\pi_{2k+6})\succeq\cdots
\succeq(\pi_{4k+1},\pi_{4k+2})\succeq\pi_{4k+3},\\
&\hspace*{7cm}\vdots\\
&\pi_{(t-1)(2k+1)+1}\succeq(\pi_{(t-1)(2k+1)+2},\pi_{(t-1)(2k+1)+3})
\succeq\cdots\succeq(\pi_{t(2k+1)-1},\pi_{t(2k+1)})\succeq\pi_{t(2k+1)+1}.
\end{align*}
\fi
Each subsequence in \eqref{buildingblock}
%\[\pi_{(i-1)(2k+1)+1}\succeq(\pi_{(i-1)(2k+1)+2},\pi_{(i-1)(2k+1)+3})\succeq\cdots\succeq(\pi_{i(2k+1)-1},\pi_{i(2k+1)})\succeq\pi_{i(2k+1)+1}\]
is a building block of the $k$-elongated partition diamond $\pi$. %Every building block of $\pi$ has exactly $2k+2$ elements. For convenience, we here require that the final building block must terminate with the zero element, i.e., $\pi_{t(2k+1)+1}=0$. 
The length of $\pi$ is defined to be the number of its building blocks, i.e., $t$.

A $k$-elongated partition diamond $\pi$ can be regarded as a geometric configuration. Each element $\pi_i$ is represented as a node, and an arrow pointing from $\pi_i$ to $\pi_j$ is interpreted as $\pi_i\geq\pi_j$. Then $\pi$ is a chain obtained by gluing its building blocks together. Figure \ref{diamond1} gives an example of the first building block, and Figure \ref{diamond2} illustrates a $k$-elongated partition diamond of length $t$.
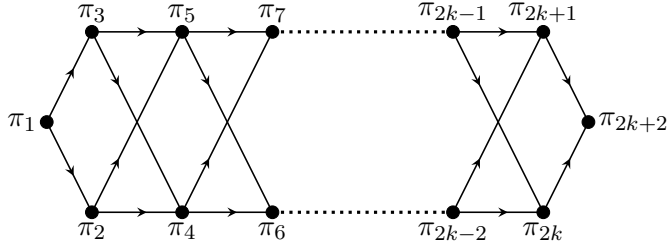
\begin{figure}[ht]
\begin{center}
\begin{tikzpicture}[scale=1.2]
\path (0,1) coordinate (1); \path (0.5,0) coordinate (2); \path (0.5,2) coordinate (3);

\path (1.5,0) coordinate (4); \path (1.5,2) coordinate (5); \path (2.5,0) coordinate (6);

\path (2.5,2) coordinate (7); \path (4.5,0) coordinate (8); \path (4.5,2) coordinate (9);

\path (5.5,0) coordinate (10);  \path (5.5,2) coordinate (11); \path (6,1) coordinate (12);

\path (1,1) coordinate (m1);  \path (2,1) coordinate (m2); \path (5,1) coordinate (m3);

\draw[line width=0.6pt, middlearrow={stealth reversed}]  (2)--(1);

\draw[line width=0.6pt, middlearrow={stealth reversed}]  (3)--(1);

\draw[line width=0.6pt, middlearrow={stealth reversed}]  (m1)--(3);

\draw[line width=0.6pt, middlearrow={stealth reversed}]  (m1)--(2);

\draw[line width=0.6pt, middlearrow={stealth reversed}]  (5)--(3);

\draw[line width=0.6pt, middlearrow={stealth reversed}]  (4)--(2);

\draw[line width=0.6pt, middlearrow={stealth reversed}]  (m2)--(5);

\draw[line width=0.6pt, middlearrow={stealth reversed}]  (m2)--(4);

\draw[line width=0.6pt, middlearrow={stealth reversed}]  (7)--(5);

\draw[line width=0.6pt, middlearrow={stealth reversed}]  (6)--(4);

\draw[line width=0.6pt, middlearrow={stealth reversed}]  (11)--(9);

\draw[line width=0.6pt, middlearrow={stealth reversed}]  (10)--(8);

\draw[line width=0.6pt, middlearrow={stealth reversed}]  (m3)--(9);

\draw[line width=0.6pt, middlearrow={stealth reversed}]  (m3)--(8);

\draw[line width=0.6pt, middlearrow={stealth reversed}]  (12)--(11);

\draw[line width=0.6pt, middlearrow={stealth reversed}]  (12)--(10);

\draw[line width=0.6pt]  (m1)--(5); \draw[line width=0.6pt]  (m1)--(4);

\draw[line width=0.6pt]  (m2)--(7); \draw[line width=0.6pt]  (m2)--(6);

\draw[line width=0.6pt]  (m3)--(11); \draw[line width=0.6pt]  (m3)--(10);

\draw[line width=1.2pt, dotted]  (7)--(9); \draw[line width=1.2pt, dotted]  (6)--(8);

\draw[fill=black] (1) circle (2pt); \draw[fill=black] (2) circle (2pt); \draw[fill=black] (3) circle (2pt);

\draw[fill=black] (4) circle (2pt); \draw[fill=black] (5) circle (2pt); \draw[fill=black] (6) circle (2pt);

\draw[fill=black] (7) circle (2pt); \draw[fill=black] (8) circle (2pt); \draw[fill=black] (9) circle (2pt);

\draw[fill=black] (10) circle (2pt); \draw[fill=black] (11) circle (2pt); \draw[fill=black] (12) circle (2pt);

\node[left] at (1) {$\pi_{1}$}; \node[below] at (2) {$\pi_2$}; \node[above] at (3) {$\pi_3$};

\node[below] at (4) {$\pi_4$}; \node[above] at (5) {$\pi_5$}; \node[below] at (6) {$\pi_6$};

\node[above] at (7) {$\pi_7$}; \node[below] at (8) {$\pi_{2k-2}$}; \node[above] at (9) {$\pi_{2k-1}$};

\node[below] at (10) {$\pi_{2k}$}; \node[above] at (11) {$\pi_{2k+1}$}; \node[right] at (12) {$\pi_{2k+2}$};
\end{tikzpicture}
\end{center}\caption{The first building block.}\label{diamond1}
\end{figure}

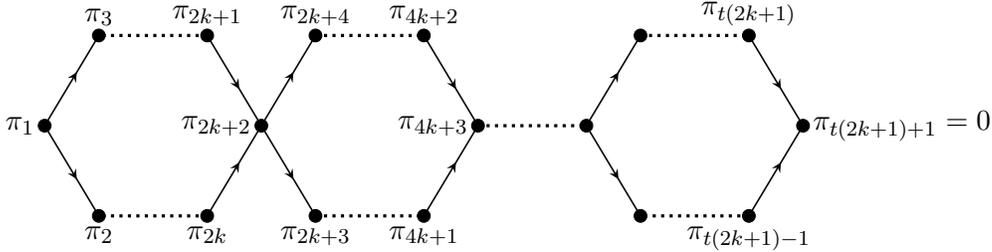
\begin{figure}[ht]
\begin{center}
\begin{tikzpicture}[scale=1.2]
\path (0,1) coordinate (1); \path (0.6,0) coordinate (2); \path (0.6,2) coordinate (3);

\path (1.8,0) coordinate (4); \path (1.8,2) coordinate (5); \path (2.4,1) coordinate (6);

\path (3,0) coordinate (7); \path (3,2) coordinate (8); \path (4.2,0) coordinate (9);

\path (4.2,2) coordinate (10); \path (4.8,1) coordinate (11);

\path (6,1) coordinate (12); \path (6.6,0) coordinate (13); \path (6.6,2) coordinate (14);

\path (7.8,0) coordinate (15); \path (7.8,2) coordinate (16); \path (8.4,1) coordinate (17);

\draw[line width=0.6pt, middlearrow={stealth reversed}]  (2)--(1);

\draw[line width=0.6pt, middlearrow={stealth reversed}]  (3)--(1);

\draw[line width=1.2pt, dotted]  (5)--(3);

\draw[line width=1.2pt, dotted]  (4)--(2);

\draw[line width=0.6pt, middlearrow={stealth reversed}]  (6)--(5);

\draw[line width=0.6pt, middlearrow={stealth reversed}]  (6)--(4);

\draw[line width=0.6pt, middlearrow={stealth reversed}]  (7)--(6);

\draw[line width=0.6pt, middlearrow={stealth reversed}]  (8)--(6);

\draw[line width=0.6pt, middlearrow={stealth reversed}]  (11)--(9);

\draw[line width=0.6pt, middlearrow={stealth reversed}]  (11)--(10);

\draw[line width=1.2pt, dotted]  (7)--(9);

\draw[line width=1.2pt, dotted]  (8)--(10);

\draw[line width=0.6pt, middlearrow={stealth reversed}]  (13)--(12);

\draw[line width=0.6pt, middlearrow={stealth reversed}]  (14)--(12);

\draw[line width=1.2pt, dotted]  (15)--(13);

\draw[line width=1.2pt, dotted]  (14)--(16);

\draw[line width=0.6pt, middlearrow={stealth reversed}]  (17)--(15);

\draw[line width=0.6pt, middlearrow={stealth reversed}]  (17)--(16);

\draw[line width=1.2pt, dotted]  (11)--(12);

\draw[fill=black] (1) circle (2pt); \draw[fill=black] (2) circle (2pt); \draw[fill=black] (3) circle (2pt);

\draw[fill=black] (4) circle (2pt); \draw[fill=black] (5) circle (2pt); \draw[fill=black] (6) circle (2pt);

\draw[fill=black] (7) circle (2pt); \draw[fill=black] (8) circle (2pt); \draw[fill=black] (9) circle (2pt);

\draw[fill=black] (10) circle (2pt); \draw[fill=black] (11) circle (2pt);

\draw[fill=black] (12) circle (2pt); \draw[fill=black] (13) circle (2pt); \draw[fill=black] (14) circle (2pt);

\draw[fill=black] (15) circle (2pt); \draw[fill=black] (16) circle (2pt); \draw[fill=black] (17) circle (2pt);

\node[left] at (1) {$\pi_1$}; \node[below] at (2) {$\pi_2$}; \node[above] at (3) {$\pi_3$};

\node[below] at (4) {$\pi_{2k}$}; \node[above] at (5) {$\pi_{2k+1}$}; \node[left] at (6) {$\pi_{2k+2}$};

\node[below] at (7) {$\pi_{2k+3}$}; \node[above] at (8) {$\pi_{2k+4}$}; \node[below] at (9) {$\pi_{4k+1}$};

\node[above] at (10) {$\pi_{4k+2}$}; \node[left] at (11) {$\pi_{4k+3}$};

\node[below] at (15) {$\pi_{t(2k+1)-1}$}; \node[above] at (16) {$\pi_{t(2k+1)}$};

\node[right] at (17) {$\pi_{t(2k+1)+1}=0$};
\end{tikzpicture}
\end{center}\caption{A $k$-elongated partition diamond of length $t$.}\label{diamond2}
\end{figure}

Let $d_k(n)$ denote the number of $k$-elongated partition diamonds of $n$. Andrews and Paule \cite{AP22} then obtained the following generating function formula for $d_k(n)$: % satisfies
\begin{align}
\sum\limits_{n=0}^\infty d_k(n)q^n=\frac{(q^2;q^2)_\infty^k}{(q;q)_\infty^{3k+1}}, \label{andrewspaule}
\end{align}
where $(a;q)_\infty=\prod\limits_{i=0}^\infty(1-aq^i)$. Since
\begin{align*}
\frac{(q^2;q^2)_\infty^k}{(q;q)_\infty^{3k+1}}=\frac{(-q;q)_\infty^k}{(q;q)_\infty^k}\frac{1}{(q;q)_\infty^{k+1}},
\end{align*}
we can interpret the above generating function using overpartitions. Recall that an \emph{overpartition} \cite{sj1} is a partition in which the final occurrence of a part may be overlined. The identity in \eqref{andrewspaule} can be restated as follows.
\begin{theorem}[Andrews--Paule \cite{AP22}]\label{gp4}
The number of $k$-elongated partition diamonds of $n$ equals the number of $(2k+1)$-tuple partitions $(\alpha^{1},\alpha^2,\ldots,\alpha^{2k+1})$ of $n$ such that $\alpha^{2},\alpha^{4},\ldots,\alpha^{2k}$ are overpartitions while the others are ordinary partitions. 
\end{theorem}

The main purpose of this paper is to study the combinatorics of Schmidt type partitions. More specifically, we give generating function formulas for Schmidt $k$-partitions and unrestricted Schmidt $k$-partitions, which generalize the result of Schmidt (Theorem~\ref{gp1}) and the results of Andrews and Paule (Theorems~\ref{gp2} and \ref{gp3}), respectively. We also introduce Schmidt type overpartitions and give analogous results. All our results are proved in a unified combinatorial way.

The rest of this paper is organized as follows. In Section 2, we establish a bijection $\Phi$ between $k\times t$ integer matrices and integer sequences of length $kt$, which palys a central role in our later combinatorial proofs. In Section 3, we focus on the combinatorics of Schmidt $k$-partitions and unrestricted Schmidt $k$-partitions connecting them to $k$-tuple partitions. In Section 4, we discuss the combinatorial properties of $k$-elongated partition diamonds and present a refinement of Theorem \ref{gp4}. In Section 5, we establish overpartition analogues of Schmidt type partition theorems.

\section{The key bijection}

In this section, we establish a bijection $\Phi$ from the set of $k\times t$ integer matrices with the sum of all entries being $n$ to the set of integer sequences $(\lambda_1,\lambda_2,\lambda_3,\ldots,\lambda_{kt})$
of length $kt$ with \[\lambda_1+\lambda_{k+1}+\lambda_{2k+1}+\cdots+\lambda_{(t-1)k+1}=n.\]

Let $A$ be a $k\times t$ matrix whose $(i,j)$-entry is $a_{i,j}$. We now introduce an operation $\phi$ on each entry $a_{i,j}$ which is define by
\begin{align*}
\phi(a_{i,j})=
\left\{
\begin{array}{ll}
\sum\nolimits_{s=1}^ka_{s,j}, & \hbox{if $i=1$;} \\[5pt]
\sum\nolimits_{s=i}^ka_{s,t}, & \hbox{if $j=t$;} \\[5pt]
\sum\nolimits_{s=i}^ka_{s,j}+\sum\nolimits_{s=1}^{i-1}a_{s,j+1}, & \hbox{otherwise.}
\end{array}
\right.
\end{align*}
For $1\leq i\leq k$ and $1\leq j\leq t$, set $\lambda_{(j-1)k+i}=\phi(a_{i,j})$. Define $\Phi(A)$ to be the sequence
\[\lambda:=(\lambda_1,\lambda_2,\ldots,\lambda_{kt}).\]

Clearly, $\Phi(A)$ is an integer sequence of length $kt$. Moreover, for $1\leq j\leq t$,
\[\lambda_{(j-1)k+1}=\phi(a_{1,j})=\sum\limits_{i=1}^ka_{i,j}.\]
We now can conclude that
\begin{align*}
\lambda_1+\lambda_{k+1}+\lambda_{2k+1}+\cdots+\lambda_{(t-1)k+1}
=\sum\limits_{j=1}^{t}\sum\limits_{i=1}^ka_{i,j},
\end{align*}
which is the sum of all entries in $A$.

For example, given a $3\times4$ matrix
\begin{align*}
A=
\begin{pmatrix}
5 & 5 & 4 & 3 \\
7 & 4 & 1 & 0 \\
-5 & -9 & 0 & 1
\end{pmatrix},
\end{align*}
then $\lambda_1=5+7-5=7$, $\lambda_2=7-5+5=7$, $\lambda_3=-5+5+4=4$, $\lambda_4=5+4-9=0$, $\lambda_5=4-9+4=-1$, $\lambda_6=-9+4+1=-4$, $\lambda_7=4+1+0=5$, $\lambda_8=1+0+3=4$, $\lambda_9=0+3+0=3$, $\lambda_{10}=3+0+1=4$, $\lambda_{11}=0+1=1$ and $\lambda_{12}=1$. Therefore,
\[\Phi(A)=(7,7,4,0,-1,-4,5,4,3,4,1,1).\]

We now show that $\Phi$ is invertible. Given an integer sequence \[\lambda=(\lambda_1,\lambda_2,\ldots,\lambda_{kt})\]
of length $kt$, we can recover the $k\times t$ matrix $A$ as follows. First, let
\begin{align*}
a_{k,t}&=\lambda_{kt},\\
a_{k-1,t}&=\lambda_{kt-1}-a_{k,t},\\
&\hspace*{0.2cm}\vdots\\
a_{1,t}&=\lambda_{kt-k+1}-a_{2,t}-a_{3,t}-\cdots-a_{k,t}.
\end{align*}
For $i$ from $k$ to $1$ and $j$ from $t-1$ to $1$, let
\begin{align*}
a_{i,j}=\left\{
\begin{array}{ll}
\lambda_{(j-1)k+1}-\sum\limits_{s=2}^{k}a_{s,j}, & \hbox{if $i$=1;} \\[10pt]
\lambda_{(j-1)k+i}-\sum\limits_{s=i+1}^{k}a_{s,j}-\sum\limits_{s=1}^{i-1}a_{s,j+1}, & \hbox{otherwise.}
\end{array}
\right.
\end{align*}
Let $A$ be the matrix with $(i,j)$-entry being $a_{i,j}$. Then $A$ is a $k\times t$ matrix and $\Phi(A)=\lambda$.

In addition, under the mapping $\Phi$, we can see that
\begin{itemize}
\item for $1\leq j\leq t-1$,
\begin{align*}
\lambda_{(j-1)k+1}-\lambda_{(j-1)k+2}
&=\phi(a_{1,j})-\phi(a_{2,j})\\
&=\sum\limits_{s=1}^ka_{s,j}-\left(\sum\limits_{s=2}^ka_{s,j}+a_{1,j+1}\right)\\
&=a_{1,j}-a_{1,j+1}.
\end{align*}
\item for $i\neq1$ and $j\neq t$,
\begin{align*}
\lambda_{(j-1)k+i}-\lambda_{(j-1)k+(i+1)}
&=\phi(a_{i,j})-\phi(a_{i+1,j})\\
&=\sum\limits_{s=i}^ka_{s,j}+\sum\limits_{s=1}^{i-1}a_{s,j+1}
-\left(\sum\limits_{s=i+1}^ka_{s,j}+\sum\limits_{s=1}^{i}a_{s,j+1}\right)\\
&=a_{i,j}-a_{i,j+1}.
\end{align*}
\item for $1\leq i\leq k$,
\begin{align*}
\lambda_{(t-1)k+i}-\lambda_{(t-1)k+(i+1)}=a_{i,t}.
\end{align*}
\end{itemize}

We now derive the following conclusion, which plays a central role in our later proofs.
\begin{theorem}\label{thmkey}
Suppose that $A$ is a nonnegative integer matrix, then
\begin{description}
\item[(a)] $\Phi(A)$ is a weakly decreasing nonnegative integer sequence if and only if each row of $A$ is weakly
decreasing;
\item[(b)] $\Phi(A)$ is a strictly decreasing nonnegative integer sequence (excluding the zero entries) if and
only if each row of $A$ is strictly decreasing and the zero entries (if exist) only occur consecutively in the last column and  extend to the last position.
\end{description}
\end{theorem}

\section{Schmidt type partition theorems}

In this section, we aim to present the general form of Theorems \ref{gp1}, \ref{gp2} and \ref{gp3}. All results are proved bijectively in a unified manner.

\subsection{Schmidt $k$-partition theorem}
\begin{theorem}\label{main1}
For given positive integers $n,k,t,r$ with $1\leq r\leq k$, let $p(n,k,t,r)$ be the number of $k$-tuple partitions $(\alpha^1,\alpha^2,\ldots,\alpha^k)$ of $n$ where each $\alpha^i$ is a strict partition, and $\alpha^1,\alpha^2,\ldots,\alpha^r$ have length $t$ and $\alpha^{r+1},\alpha^{r+2},\ldots,\alpha^{k}$ have length $t-1$. Let $q(n,k,t,r)$ be the number of Schmidt $k$-partitions of $n$ with $(t-1)k+r$ parts. Then we have \[p(n,k,t,r)=q(n,k,t,r).\]
\end{theorem}
\begin{proof}
Let $(\alpha^1,\alpha^2,\ldots,\alpha^k)$ be a $k$-tuple partition counted by $p(n,k,t,r)$ with
\[\alpha^i=(\alpha^i_1,\alpha^i_2,\ldots,\alpha^i_t)\]
where $\alpha^i_t=0$ for $r+1\leq i\leq k$.

We now construct a $k\times t$ matrix $A=(a_{i,j})$ with $a_{i,j}=\alpha^i_j$, i.e., the $i$-th row of $A$ coincides with the partition $\alpha^i$. Applying the algorithm $\Phi$ to $A$, we get a nonnegative integer sequence
\[(\lambda_1,\lambda_2,\ldots,\lambda_{kt}).\]

We claim that $\lambda_1>\lambda_2>\cdots>\lambda_{(t-1)k+r}>0$ and
\[\lambda_{(t-1)k+(r+1)}=\lambda_{(t-1)k+(r+2)}=\cdots=\lambda_{kt}=0.\]
It follows from the definition of $(\alpha^1,\alpha^2,\ldots,\alpha^k)$ that the last $k-r$ entries in the last column of $A$ are equal to zero and all the other entries of $A$ are positive. Therefore, all $\lambda_i$ are greater than $0$ except $\lambda_{(t-1)k+(r+1)},\lambda_{(t-1)k+(r+2)},\ldots,\lambda_{kt}$. Since each $\alpha^i$ is a strict partition, then $A$ is strictly decreasing in rows, which ensures that
\[\lambda_1>\lambda_2>\cdots>\lambda_{(t-1)k+r}\]
by Theorem \ref{thmkey}.

It is clear that $\lambda_1+\lambda_{k+1}+\cdots+\lambda_{(t-1)k+1}=n$. Thus, $(\lambda_1, \lambda_2,\ldots\lambda_{(t-1)k+r})$ is a Schmidt $k$-partition of $n$ with exactly $(t-1)k+r$ parts.

Conversely, given a Schmidt $k$-partition $(\lambda_1,\lambda_2,\ldots,\lambda_{(t-1)k+r})$ enumerated by $q(n,k,t,r)$, we obtain a nonnegative integer sequence of length $kt$
\[(\lambda_1,\lambda_2,\ldots,\lambda_{(t-1)k+r},0,0,\ldots,0)\]
by appending $k-r$ zeroes. Then the inverse of this sequence under the mapping $\Phi$ is a $k\times t$ nonnegative matrix $A$ whose zero entries are all collected in the last $k-r$ positions of the last column. Due to Theorem \ref{thmkey}, we can view each row of $A$ as a strict partition. Furthermore, the first $k$ partitions have length $t$ and the others have length $t-1$.
\end{proof}

\begin{corollary}\label{c1main1}
The number of partitions of $n$ with Durfee square size $t$ is equal to the number of Schmidt $2$-partitions of $n$ with $2t$ or $2t-1$ parts.
\end{corollary}
\begin{proof}
Given a partition $\lambda$ of $n$ with Durfee square size $t$, we decompose $\lambda$ into two strict partitions as follows. Let $\alpha$ be the partition represented below the main diagonal of the Ferrers diagram, where we read the parts from left to right and the main diagonal is included. Let $\beta$ be the partition located to the right of the main diagonal of the Ferrers diagram, where we read from top to bottom. Then $\alpha$ is a strict partition of length $t$, and $\beta$ is a strict partition whose length is $t$ or $t-1$ depending on the $t$-th part of $\lambda$ is greater than $t$ or equal to $t$.

If the length of $\beta$ is $t$, then the bijection $\Phi$ maps the bipartition $(\alpha,\beta)$ of $n$ to a Schmidt $2$-partition of $n$ with length $2t$; otherwise, it corresponds to a Schmidt $2$-partition of length $2t-1$.
\end{proof}

\begin{remark}
It is easy to see that Corollary \ref{c1main1} is a refinement of Theorem \ref{gp1} according to the size of Durfee square.
\end{remark}

To prove Theorem \ref{gp3}, we need a simple auxiliary result.
\begin{lemma}\label{l1gp3}
The number of triple partitions $(\alpha,\beta,\gamma)$ of $n$ such that $\alpha$ is a $2$-distinct partition of length $t$, $\beta$ is a strict partition of length $t$ and $\gamma$ is an ordinary partition of length at most $t$ is equal to the number of triple partitions $(\mu,\nu,\omega)$ of $n$ such that $\mu$ is a strict partition of length $t$, $\nu$ is a strict partition of length $t$ and $\omega$ is a strict partition of length $t$ or $t-1$.
\end{lemma}
\begin{proof}
Given a triple partition $(\alpha,\beta,\gamma)$, let $\mu$ be the partition obtained from $\alpha$ by removing the triangle partition $(t-1,t-2,\ldots,1,0)$, and define $\omega$ to be the partition obtained by adding the triangle partition $(t-1,t-2,\ldots,1,0)$ to $\gamma$. Set $\nu=\beta$. This operation is clearly invertible.

It is easy to see that $\mu$ is a strict partition of length $t$ since $\alpha$ is a $2$-distinct partition of length $t$. If the length of $\gamma$ is $t$, then $\omega$ is a strict partition of length $t$; otherwise, $\omega$ is a strict partition of length $t-1$.
\end{proof}

As a consequence of Lemma \ref{l1gp3} and Theorem \ref{main1}, we can state Theorem \ref{gp3} more precisely as follows.
\begin{corollary}
The number of triple partitions $(\alpha,\beta,\gamma)$ of $n$ where $\alpha$ is a $2$-distinct partition of length $t$, $\beta$ is a strict partition of length $t$ and $\gamma$ is an ordinary partition of length $t$ equals the number of Schmidt $3$-partitions of $n$ with $3t$ parts; if the length of $\gamma$ is less than $t$, then it is equal to the number of Schmidt $3$-partitions of $n$ with $3t-1$ parts.
\end{corollary}

\subsection{Unrestricted Schmidt $k$-partition theorem}
\begin{theorem}\label{main2}
For given positive integers $n,k,t,r$ with $1\leq r\leq k$, let $f(n,k,t,r)$ be the number of $k$-tuple partitions $(\alpha^1,\alpha^2,\ldots,\alpha^k)$ of $n$ where $\max\{\ell(\alpha^1),\ell(\alpha^2),\ldots,\ell(\alpha^k)\}=t$ and $r$ is the largest integer such that $\ell(\alpha^r)=t$. Let $g(n,k,t,r)$ be the number of unrestricted Schmidt $k$-partitions of $n$ with length being equal to $(t-1)k+r$. Then \[f(n,k,t,r)=g(n,k,t,r).\]
\end{theorem}
\begin{proof}
Let $(\alpha^1,\alpha^2,\ldots,\alpha^k)$ be a $k$-tuple partition of $n$ counted by $f(n,k,t,r)$. For each $i$, we can view $\alpha^i$ as a partition of length $t$ with zeroes permitted. Define $A=(a_{i,j})$ to be the $k\times t$ matrix whose $i$-th row corresponds to the partition $\alpha^i$. Applying the algorithm $\Phi$ to $A$, we get a nonzero sequence
\[\Phi(A)=(\lambda_1,\lambda_2,\ldots,\lambda_{kt})\]
of length $kt$.

Since $a_{r,t}>0$ and $a_{r+1,t}=a_{r+2,t}=\cdots=a_{k,t}=0$, we have $\lambda_{(t-1)k+r}>0$ and \[\lambda_{(t-1)k+r+1}=\lambda_{(t-1)k+r+2}=\cdots=\lambda_{kt}=0.\]
It follows from Theorem \ref{thmkey} that $\lambda$ is a weakly decreasing sequence since $A$ is weakly decreasing in rows. In addition, it is clear that $\lambda_{1}+\lambda_{k+1}+\lambda_{2k+1}+\cdots+\lambda_{(t-1)k+1}=n$.

We now can see that
\[(\lambda_1,\lambda_2,\ldots,\lambda_{(t-1)k+r})\]
is a Schmidt $k$-partition of $n$ with exactly $(t-1)k+r$ parts.

The above process is invertible, whose proof is very similar to that of Theorem \ref{main1} and is omitted here.
\end{proof}

\begin{remark}
Since the largest part of each $\alpha^i$ is recorded in the first column of $A$, and \[\max\{\ell(\alpha^1),\ell(\alpha^2),\ldots,\ell(\alpha^k)\}\]
is equal to the number of columns of $A$, we can derive Theorem 2 in \cite{Ji21} easily and immediately from our proof.
\end{remark}

If we sum over $t\geq1$ and $1\leq r\leq k$, we can obtain the following result immediately.
\begin{corollary}\label{c1main2}
The number of unrestricted Schmidt $k$-partitions of $n$ is equal to the number of $k$-tuple partitions of $n$.
\end{corollary}
\begin{remark}
Setting $k=2$, Corollary \ref{c1main2} reduces to Theorem \ref{gp2}.
\end{remark}

\section{Partition diamonds}

\begin{theorem}\label{main3}
For given integers $n,k,t,r$ with $1\leq r\leq 2k+1$, let $h(n,k,t,r)$ be the number of $(2k+1)$-tuple partitions $(\alpha^{1},\alpha^2,\ldots,\alpha^{2k+1})$ of $n$ where $\alpha^{2},\alpha^{4},\ldots,\alpha^{2k}$ are overpartitions while the others are ordinary partitions, $\max\{\ell(\alpha^1),\ell(\alpha^2),\ldots,\ell(\alpha^{2k+1})\}=t$ and $r$ is the largest integer such that $\ell(\alpha^r)=t$. Let $d(n,k,t,r)$ be the number of $k$-elongated partition diamonds of $n$ with length $t$ and $(t-1)(2k+1)+r$ nonzero elements. Then $h(n,k,t,r)=d(n,k,t,r)$.
\end{theorem}
\begin{proof}
Let $(\alpha^{1},\alpha^2,\ldots,\alpha^{2k+1})$ be a $(2k+1)$-tuple partition counted by $h(n,k,t,r)$. We can regard each $\alpha^i$ as a partition of length $t$ by adding some zero parts. Define $A=(a_{i,j})$ to be the $(2k+1)\times t$ matrix whose $i$-th row consists of the partition $\alpha^i$.

The algorithm $\Phi$ maps $A$ to a nonnegative integer sequence
\[\lambda=(\lambda_1,\lambda_2,\ldots,\lambda_{(2k+1)t})\]
of length $(2k+1)t$. Since $\alpha^r$ is the last partition with exactly $t$ nonzero parts, we can derive that
$\lambda_{(t-1)(2k+1)+r}>0$ and
\[\lambda_{(t-1)(2k+1)+(r+1)}=\lambda_{(t-1)(2k+1)+(r+2)}=\cdots=\lambda_{t(2k+1)}=0.\]
Meanwhile, since the matrix $A$ is weakly decreasing in rows, we have
\[\lambda_1\geq\lambda_2\geq\cdots\geq\lambda_{(t-1)(2k+1)+r}>0.\]

We now define a $k$-elongated partition diamond $\pi=(\pi_1,\pi_2,\ldots,\pi_{(2k+1)t+1})$ of length $t$ where
\begin{description}
\item[(1)]$\pi_{(j-1)(2k+1)+1}=\lambda_{(j-1)(2k+1)+1}$ for $1\leq j\leq t$, and $\pi_{t(2k+1)+1}=0$;
\item[(2)]for $1\leq i\leq k$ and $1\leq j\leq t$,
\[(\pi_{(j-1)(2k+1)+2i},\pi_{(j-1)(2k+1)+2i+1})=(\lambda_{(j-1)(2k+1)+2i+1},\lambda_{(j-1)(2k+1)+2i})\]
if the $j$-th part of $\alpha^{2i}$ is overlined; otherwise,
\[(\pi_{(j-1)(2k+1)+2i},\pi_{(j-1)(2k+1)+2i+1})=(\lambda_{(j-1)(2k+1)+2i},\lambda_{(j-1)(2k+1)+2i+1}).\]
\end{description}

Because $\lambda$ is a weakly decreasing sequence, we can see that $\pi$ satisfies
\[\pi_1\succeq(\pi_2,\pi_3)\succeq\cdots\succeq(\pi_{2k},\pi_{2k+1})
\succeq\pi_{2k+2}\succeq(\pi_{2k+3},\pi_{2k+4})\succeq\cdots\succeq\pi_{t(2k+1)+1}.\]
Obviously,
\[\pi_{1}+\pi_{2k+2}+\pi_{4k+3}+\cdots+\pi_{(t-1)(2k+1)+1}
=\lambda_{1}+\lambda_{2k+2}+\lambda_{4k+3}+\cdots+\lambda_{(t-1)(2k+1)+1}=n.\]
Therefore, $\pi$ is a $k$-elongated partition diamond of $n$ with length $t$ and $(t-1)(2k+1)+r$ nonzero elements.

In an overpartition, the final occurrence of a part may be overlined. Thus, if the $j$-th part of $\alpha^{2i}$ is overlined, then it is strictly greater than the $(j+1)$-st part of $\alpha^{2i}$. Consequently, $\lambda_{(j-1)(2k+1)+2i}>\lambda_{(j-1)(2k+1)+2i+1}$, and hence $\pi_{(j-1)(2k+1)+2i}<\pi_{(j-1)(2k+1)+2i+1}$. This property is very important to the inverse.

Conversely, given a $k$-elongated partition diamond $\pi=(\pi_1,\pi_2,\ldots,\pi_{t(2k+1)+1})$ counted by $d(n,k,t,r)$, we first delete $\pi_{t(2k+1)+1}$ and rearrange the parts of $\pi$ into a weakly decreasing sequence $\lambda=(\lambda_1,\lambda_2,\ldots,\lambda_{t(2k+1)})$ with
\[\lambda_1\geq\lambda_2\geq\cdots\geq\lambda_{t(2k+1)}.\]
Then applying the inverse of $\Phi$ to $\lambda$ produces a $(2k+1)\times t$ matrix $A=(a_{i,j})$. Finally, for each $1\leq i\leq k$ and $1\leq j\leq t$, we add an overline to $a_{2i,j}$ if $\pi_{(j-1)(2k+1)+2i}$ is strictly less than $\pi_{(j-1)(2k+1)+2i+1}$. Thinking of each row of $A$ as a partition if we ignore the zero entries, we obtain a $(2k+1)$-tuple partition counted by $h(n,k,t,r)$.
\end{proof}

\begin{example}
Let $\alpha=(\alpha^1,\alpha^2,\alpha^3,\alpha^4,\alpha^5)$ be a $5$-tuple of partitions counted by $h(73,2,4,3)$ where
\[\alpha^1=(5,2),\alpha^2=(7,4,\overline{4},2),\alpha^3=(11,3,2,1),\alpha^4=(\overline{6},5,\overline{3}),
\alpha^5=(8,8,2).\]
Regarding $\alpha^i$ as the $i$-th row of the matrix $A$, we obtain
\begin{align*}
A=
\begin{pmatrix}
5 & 2 & 0 & 0 \\
7 & 4 & \overline{4} & 2 \\
11 & 3 & 2 & 1\\
\overline{6} & 5 & \overline{3} & 0 \\
8 & 8 & 2 & 0
\end{pmatrix}.
\end{align*}
Performing the algorithm $\Phi$ produces a nonnegative integer sequence of length $20$
\[\lambda=(37,34,31,23,22,22,20,20,19,17,11,11,9,8,5,3,3,1,0,0).\]
Finally, we obtain a $2$-elongated partition diamond $\pi=(\pi_1,\pi_2,\ldots,\pi_{21})$ where
\begin{align*}
\begin{array}{lll}
\pi_1=\lambda_1=37,&(\pi_{2},\pi_{3})=(\lambda_{2},\lambda_{3})=(34,31),
&(\pi_{4},\pi_{5})=(\lambda_{5},\lambda_{4})=(22,23),\\
\pi_{6}=\lambda_{6}=22,&(\pi_{7},\pi_{8})=(\lambda_{7},\lambda_{8})=(20,20),
&(\pi_{9},\pi_{10})=(\lambda_{9},\lambda_{10})=(19,17),\\
\pi_{11}=\lambda_{11}=11,&(\pi_{12},\pi_{13})=(\lambda_{13},\lambda_{12})=(9,11),
&(\pi_{14},\pi_{15})=(\lambda_{15},\lambda_{14})=(5,8),\\
\pi_{16}=\lambda_{16}=3,&(\pi_{17},\pi_{18})=(\lambda_{17},\lambda_{18})=(3,1),
&(\pi_{19},\pi_{20})=(\lambda_{19},\lambda_{20})=(0,0),\\
\pi_{21}=0. &{} &{}
\end{array}
\end{align*}
There are $4$ building blocks and $18$ nonzero elements in $\pi$, which is depicted in Figure \ref{figexa}.
\begin{figure}[ht]
\begin{center}
\begin{tikzpicture}[scale=1.2]
\path (0,1) coordinate (1); \path (0.6,2) coordinate (2); \path (0.6,0) coordinate (3);

\path (1.8,2) coordinate (4); \path (1.8,0) coordinate (5); \path (2.4,1) coordinate (6);

\path (3,2) coordinate (7); \path (3,0) coordinate (8); \path (4.2,2) coordinate (9);

\path (4.2,0) coordinate (10); \path (4.8,1) coordinate (11);

\path (5.4,2) coordinate (12); \path (5.4,0) coordinate (13);

\path (6.6,2) coordinate (14); \path (6.6,0) coordinate (15);

\path (7.2,1) coordinate (16); \path (7.8,2) coordinate (17); \path (7.8,0) coordinate (18);

\path (9,2) coordinate (19); \path (9,0) coordinate (20); \path (9.6,1) coordinate (21);

\path (1.2,1) coordinate (n1); \path (3.6,1) coordinate (n2);

\path (6,1) coordinate (n3); \path (8.4,1) coordinate (n4);

\draw[line width=0.6pt, middlearrow={stealth reversed}]  (2)--(1);

\draw[line width=0.6pt, middlearrow={stealth reversed}]  (3)--(1);

\draw[line width=0.6pt, middlearrow={stealth reversed}]  (5)--(3);

\draw[line width=0.6pt, middlearrow={stealth reversed}]  (4)--(2);

\draw[line width=0.6pt, middlearrow={stealth reversed}]  (6)--(5);

\draw[line width=0.6pt, middlearrow={stealth reversed}]  (6)--(4);

%%%%%%%%%%%%%%%%%%%%%%%%%%%%%%%%%%%%%%%%%%%%%%%%%%%%%%%%%%%%%
\draw[line width=0.6pt, middlearrow={stealth reversed}]  (n1)--(2);

\draw[line width=0.6pt, middlearrow={stealth reversed}]  (n1)--(3);

\draw[line width=0.6pt]  (n1)--(4);

\draw[line width=0.6pt]  (n1)--(5);

%%%%%%%%%%%%%%%%%%%%%%%%%%%%%%%%%%%%%%%%%%%%%%%%%%%%%%%%%%%%%
\draw[line width=0.6pt, middlearrow={stealth reversed}]  (7)--(6);

\draw[line width=0.6pt, middlearrow={stealth reversed}]  (8)--(6);

\draw[line width=0.6pt, middlearrow={stealth reversed}]  (11)--(9);

\draw[line width=0.6pt, middlearrow={stealth reversed}]  (11)--(10);

\draw[line width=0.6pt, middlearrow={stealth reversed}]  (9)--(7);

\draw[line width=0.6pt, middlearrow={stealth reversed}]  (10)--(8);

%%%%%%%%%%%%%%%%%%%%%%%%%%%%%%%%%%%%%%%%%
\draw[line width=0.6pt, middlearrow={stealth reversed}]  (n2)--(7);

\draw[line width=0.6pt, middlearrow={stealth reversed}]  (n2)--(8);

\draw[line width=0.6pt]  (n2)--(9);

\draw[line width=0.6pt]  (n2)--(10);
%%%%%%%%%%%%%%%%%%%%%%%%%%%%%%%%%%%%%%%%%%%
\draw[line width=0.6pt, middlearrow={stealth reversed}]  (12)--(11);

\draw[line width=0.6pt, middlearrow={stealth reversed}]  (13)--(11);

\draw[line width=0.6pt, middlearrow={stealth reversed}]  (14)--(12);

\draw[line width=0.6pt, middlearrow={stealth reversed}]  (15)--(13);

%%%%%%%%%%%%%%%%%%%%%%%%%%%%%%%%%%%%%%%%
\draw[line width=0.6pt, middlearrow={stealth reversed}]  (n3)--(12);

\draw[line width=0.6pt, middlearrow={stealth reversed}]  (n3)--(13);

\draw[line width=0.6pt]  (n3)--(14);

\draw[line width=0.6pt]  (n3)--(15);
%%%%%%%%%%%%%%%%%%%%%%%%%%%%%%%%%%%%%%%%%

\draw[line width=0.6pt, middlearrow={stealth reversed}]  (16)--(14);

\draw[line width=0.6pt, middlearrow={stealth reversed}]  (16)--(15);

\draw[line width=0.6pt, middlearrow={stealth reversed}]  (17)--(16);

\draw[line width=0.6pt, middlearrow={stealth reversed}]  (17)--(16);

\draw[line width=0.6pt, middlearrow={stealth reversed}]  (18)--(16);

\draw[line width=0.6pt, middlearrow={stealth reversed}]  (19)--(17);

\draw[line width=0.6pt, middlearrow={stealth reversed}]  (20)--(18);

%%%%%%%%%%%%%%%%%%%%%%%%%%%%%%%%%%%%%%%%
\draw[line width=0.6pt, middlearrow={stealth reversed}]  (n4)--(17);

\draw[line width=0.6pt, middlearrow={stealth reversed}]  (n4)--(18);

\draw[line width=0.6pt]  (n4)--(19);

\draw[line width=0.6pt]  (n4)--(20);
%%%%%%%%%%%%%%%%%%%%%%%%%%%%%%%%%%%%%%%%%

\draw[line width=0.6pt, middlearrow={stealth reversed}]  (21)--(19);

\draw[line width=0.6pt, middlearrow={stealth reversed}]  (21)--(20);

\draw[fill=black] (1) circle (2pt); \draw[fill=black] (2) circle (2pt); \draw[fill=black] (3) circle (2pt);

\draw[fill=black] (4) circle (2pt); \draw[fill=black] (5) circle (2pt); \draw[fill=black] (6) circle (2pt);

\draw[fill=black] (7) circle (2pt); \draw[fill=black] (8) circle (2pt); \draw[fill=black] (9) circle (2pt);

\draw[fill=black] (10) circle (2pt); \draw[fill=black] (11) circle (2pt);

\draw[fill=black] (12) circle (2pt); \draw[fill=black] (13) circle (2pt); \draw[fill=black] (14) circle (2pt);

\draw[fill=black] (15) circle (2pt); \draw[fill=black] (16) circle (2pt); \draw[fill=black] (17) circle (2pt);

\draw[fill=black] (18) circle (2pt); \draw[fill=black] (19) circle (2pt); \draw[fill=black] (20) circle (2pt);

\draw[fill=black] (21) circle (2pt);

\node[left] at (1) {$37$}; \node[above] at (2) {$31$}; \node[below] at (3) {$34$};

\node[above] at (4) {$23$}; \node[below] at (5) {$22$}; \node[left] at (6) {$22$};

\node[above] at (7) {$20$}; \node[below] at (8) {$20$}; \node[above] at (9) {$17$};

\node[below] at (10) {$19$}; \node[left] at (11) {$11$};

\node[above] at (12) {$11$}; \node[below] at (13) {$9$};

\node[above] at (14) {$8$}; \node[below] at (15) {$5$};

\node[left] at (16) {$3$};

\node[above] at (17) {$1$}; \node[below] at (18) {$3$};

\node[above] at (19) {$0$}; \node[below] at (20) {$0$};

\node[right] at (21) {$0$};

\end{tikzpicture}
\end{center}\caption{A $2$-elongated partition diamond of $73$.}\label{figexa}
\end{figure}
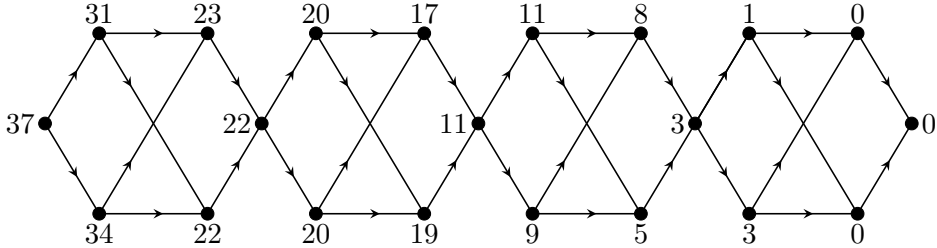
\end{example}

\section{Schmidt type overpartition theorems}

In the treatment of $k$-elongated partition diamonds, we use the tuple partitions in which ordinary partitions and overpartitions appear alternatively. It is natural to ask what results will be obtained if only overpartitions are allowed. In this section, we establish two results on overpartitions which can be viewed as the analogue of the main results in Section 3.

\subsection{Schmidt $k$-overpartition theorem}

An overpartition $(\lambda_1,\lambda_2,\ldots,\lambda_l)$ is called a \emph{strict overpartition} if $\lambda_1>\lambda_2>\cdots>\lambda_l$, and $\lambda_i$ may be overlined only if $\lambda_{i}-\lambda_{i+1}\geq2$ or $\lambda_i=\lambda_l$. A \emph{Schmidt $k$-overpartition} is a strict overpartition $(\lambda_1,\lambda_2,\lambda_3,\ldots)$ with $\lambda_1+\lambda_{k+1}+\lambda_{2k+1}+\cdots=n$. 

\begin{theorem}\label{main4}
For given positive integers $n,k,t,r,s$ with $1\leq r\leq k$, let $\overline{p}(n,k,t,r,s)$ be the number of $k$-tuple strict overpartitions $(\alpha^1,\alpha^2,\ldots,\alpha^k)$ of $n$ with totally $s$ overlined parts where $\alpha^1,\alpha^2,\ldots,\alpha^r$ have length $t$ and $\alpha^{r+1},\alpha^{r+2},\ldots,\alpha^{k}$ have length $t-1$, and each overlined part of $\alpha^i$ must be at least $2$ for all $i\neq r$. Let $\overline{q}(n,k,t,r,s)$ be the number of Schmidt $k$-overpartitions of $n$ with $(t-1)k+r$ parts, and $s$ parts of which are overlined. Then \[\overline{p}(n,k,t,r,s)=\overline{q}(n,k,t,r,s).\]
\end{theorem}
\begin{proof}
Let $\alpha=(\alpha^1,\alpha^2,\ldots,\alpha^k)$ be a $k$-tuple strict  overpartition counted by $\overline{p}(n,k,t,r,s)$. Ignoring the overlines, we can treat $\alpha$ as a $k$-tuple partition counted by $p(n,k,t,r)$. It follows from the proof of Theorem \ref{main1} that $\alpha$ corresponds to a Schmidt $k$-partition \[\lambda=(\lambda_1,\lambda_2,\ldots,\lambda_{(t-1)k+r})\] 
counted by $q(n,k,t,r)$.

For $1\leq i\leq k$ and $1\leq j\leq t$, we overline the part $\lambda_{(j-1)k+i}$ if and only if the $j$-th part of $\alpha^i$ is overlined. We now can view $\lambda$ as an overpartition. Because $\lambda_{(j-1)k+i}-\lambda_{(j-1)k+i+1}$ equals the difference between the $j$-th part and $(j+1)$-st part of $\alpha^i$, we can conclude that $\lambda_{(j-1)k+i}-\lambda_{(j-1)k+i+1}\geq2$ if $\lambda_{(j-1)k+i}$ is overlined. Meanwhile, $\lambda_{(t-1)k+i}-\lambda_{(t-1)k+i+1}$ equals the last part of $\alpha^i$. Thus, if $\lambda_{(t-1)k+i}$ is overlined, then either the last part of $\alpha^i$ is greater than one or $i=r$, which implies that either $\lambda_{(t-1)k+i}-\lambda_{(t-1)k+i+1}\geq 2$ or $\lambda_{(t-1)k+i}$ is the last nonzero part of $\lambda$. Therefore, $\lambda$ is a Schmidt $k$-overpartition.

It is clear that there are totally $s$ overlined parts in $\lambda$ since the total number of overlined parts of $\alpha$ is $s$.
\end{proof}

\begin{example}
Let $\alpha=(\alpha^1,\alpha^2,\alpha^3)$ be a triple strict overpartition counted by $\overline{p}(79,3,5,2,6)$ where
\[\alpha^1=(9,6,\overline{4},2,1),\alpha^2=(\overline{13},8,\overline{7},3,\overline{1}),\alpha^3=(11,\overline{7},5,\overline{2}).\]
Regarding $\alpha^i$ as the $i$-th row of the matrix $A$, we have
\begin{align*}
A=
\begin{pmatrix}
9 & 6 & \overline{4} & 2 & 1 \\
\overline{13} & 8 & \overline{7} & 3 & \overline{1}\\
11 & \overline{7} & 5 & \overline{2} & 0\\
\end{pmatrix}.
\end{align*}
The algorithm $\Phi$ produces a strictly decreasing nonnegative integer sequence of length $15$
\[\lambda=(33,30,25,21,19,18,16,14,10,7,6,4,2,1,0).\]
Adding overlines and deleting the zeroes, we obtain a Schmidt $3$-overpartition
\[(33,\overline{30},25,21,19,\overline{18},\overline{16},\overline{14},10,7,6,\overline{4},2,\overline{1}),\]
which is counted by $\overline{q}(79,3,5,2,6)$.
\end{example}

Since the final occurrence of a part may be overlined, an overpartition corresponds to an ordinary Ferrers graph in which a corner with no cell below it may be marked. The Durfee square of an overpartition is referred to the Durfee square of its marked Ferrers graph. For example, the overpartition $(7,\overline{7},5,4,\overline{4},2,\overline{1})$ is represented in Figure \ref{figover}.
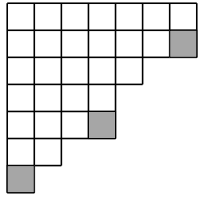
\begin{figure}[ht]
\begin{center}
\begin{tikzpicture}[scale=1.2]

%%%%%%%%%%%%%%%%% lambda %%%%%%%%%%%%%%%%%%%%%
%%%%%%%%%%%%%%%%% row %%%%%%%%%%%%%%%%%%%%%%
\draw[line width=0.6pt]  (0,2.1)--(2.1,2.1);
\draw[line width=0.6pt]  (0,1.8)--(2.1,1.8);
\draw[line width=0.6pt]  (0,1.5)--(2.1,1.5);
\draw[line width=0.6pt]  (0,1.2)--(1.5,1.2);
\draw[line width=0.6pt]  (0,0.9)--(1.2,0.9);
\draw[line width=0.6pt]  (0,0.6)--(1.2,0.6);
\draw[line width=0.6pt]  (0,0.3)--(0.6,0.3);
\draw[line width=0.6pt]  (0,0)--(0.3,0);
%%%%%%%%%%%%%%%% column %%%%%%%%%%%%%%%%%%%%%%%%
\draw[line width=0.6pt]  (0,2.1)--(0,0);
\draw[line width=0.6pt]  (0.3,2.1)--(0.3,0);
\draw[line width=0.6pt]  (0.6,2.1)--(0.6,0.3);
\draw[line width=0.6pt]  (0.9,2.1)--(0.9,0.6);
\draw[line width=0.6pt]  (1.2,2.1)--(1.2,0.6);
\draw[line width=0.6pt]  (1.5,2.1)--(1.5,1.2);
\draw[line width=0.6pt]  (1.8,2.1)--(1.8,1.5);
\draw[line width=0.6pt]  (2.1,2.1)--(2.1,1.5);

\draw[fill=gray!70] (0,0) rectangle (0.3,0.3);
\draw[fill=gray!70] (0.9,0.6) rectangle (1.2,0.9);
\draw[fill=gray!70] (1.8,1.5) rectangle (2.1,1.8);

\end{tikzpicture}
\end{center}\caption{The Ferrers graph of an overpartition.}\label{figover}
\end{figure}

\begin{corollary}\label{c1main4}
The number of overpartitions of $n$ with $s$ overlined parts and Durfee square size $t$ is equal to the number of Schmidt $2$-overpartitions of $n$ with $s$ parts being overlined and length $2t$ or $2t-1$.
\end{corollary}
\begin{proof}
Assuming that $\lambda$ is an overpartiton with Durfee square size $t$, we decompose $\lambda$ into two strict overpartitions $\alpha$ and $\beta$ just like what we do in the proof of Corollary \ref{c1main1}. A part of $\alpha$ or $\beta$ is overlined if the last cell of the corresponding column or row is marked. Then $\alpha$ is a strict overpartition of length $t$, and $\beta$ is a strict overpartition whose length is $t$ or $t-1$ depending on the $t$-th part of $\lambda$ is greater than $t$ or equal to $t$.

If the last part of $\alpha$ is equal to $1$, then it can be overlined only if the length of $\beta$ is $t-1$. Thus, $(\alpha,\beta)$ is a couple strict overpartiton counted by $\overline{p}(n,2,t,r,s)$. 

We now can conclude the desired result from Theorem \ref{main4}.
\end{proof}

Summing up $t$ over all nonnegative integers in Corollary \ref{c1main4}, we have the following result.
\begin{corollary}\label{c2main4}
The number of overpartitions of $n$ with $s$ overlined parts equals the number of Schmidt $2$-overpartitions of $n$ with $s$ parts being overlined.
\end{corollary}

\subsection{Unrestricted Schmidt $k$-overpartition theorem}

An overpartition $(\lambda_1,\lambda_2,\lambda_3,\ldots)$ is called an \emph{unrestricted Schmidt $k$-overpartition} if \[\lambda_1+\lambda_{k+1}+\lambda_{2k+1}+\cdots=n.\] 

\begin{theorem}\label{main5}
For given positive integers $n,k,t,r,s$ where $1\leq r\leq k$, let $\overline{f}(n,k,t,r,s)$ be the number of $k$-tuple overpartitions $(\alpha^1,\alpha^2,\ldots,\alpha^k)$ of $n$ with $s$ parts being overlined in which  $\max\{\ell(\alpha^1),\ell(\alpha^2),\ldots,\ell(\alpha^k)\}=t$ and $r$ is the largest integer such that $\ell(\alpha^r)=t$. Let $\overline{g}(n,k,t,r,s)$ be the number of unrestricted Schmidt $k$-overpartitions of $n$ with $s$ parts being overlined and length being equal to $(t-1)k+r$. Then \[\overline{f}(n,k,t,r,s)=\overline{g}(n,k,t,r,s).\]
\end{theorem}

\begin{proof}
Let $\alpha=(\alpha^1,\alpha^2,\ldots,\alpha^k)$ be a $k$-tuple overpartition counted by $\overline{f}(n,k,t,r,s)$. We first regard $\alpha$ as an ordinary $k$-tuple partition and then perform the same operation used in the proof of Theorem \ref{main2}. We obtain an unrestricted Schmidt $k$-partition
\[\lambda=(\lambda_1,\lambda_2,\ldots,\lambda_{(t-1)k+r})\]
counted by $g(n,k,t,r)$.

We now covert $\lambda$ into an overpartition by adding overlines as follows. The part $\lambda_{(j-1)k+i}$ is overlined if and only if the $j$-th part of $\alpha^{i}$ is overlined. Obviously, the number of overlined parts remains unchanged. Thus, we get an unrestricted Schmidt $k$-overpartition counted by $\overline{g}(n,k,t,r,s)$.
\end{proof}

\begin{corollary}
The number of $k$-tuple overpartitions of $n$ with $s$ overlined parts is equal to the number of unrestricted Schmidt $k$-overpartitions of $n$ with $s$ parts being overlined.
\end{corollary}

%\section*{Acknowledgments}
%The author would like to thank Prof. Ae Ja Yee for helpful comments and encouragement.

\end{document}